\documentclass[11pt]{amsart}

\usepackage{latexsym}
\usepackage{amssymb}
\usepackage{amsmath}
\usepackage{amsthm}
\usepackage{amsfonts}

\usepackage{pictexwd,dcpic}

\usepackage{graphicx}

\newtheorem{theorem}{Theorem}[section]
\newtheorem{proposition}[theorem]{Proposition}

\newtheorem{lemma}[theorem]{Lemma}
\newtheorem{claim}[theorem]{Claim}

\theoremstyle{definition}

\theoremstyle{remark}
\newtheorem{remark}[theorem]{Remark}

\numberwithin{equation}{section}

\newtheorem*{ack}{Acknowledgements}

\numberwithin{equation}{section}

\newcommand{\subind}[1]{_{ \scriptscriptstyle #1}  }
 
\newcommand{\supind}[1]{^ { \scriptscriptstyle #1}  }

\DeclareMathOperator{\Tub}{Tub}

\newcommand{\tub}{\ensuremath{\mathrm{Tub}}}

\newcommand{\F}{\ensuremath{\mathcal{F}}}

\newcommand{\DR}{\ensuremath{r}}

\newcommand{\metric}{\ensuremath{\mathsf{g}}}

\newcommand{\sliceindice}{\ensuremath{\alpha}}

\newcommand{\normalindice}{\ensuremath{l}}

\newcommand{\totalFindice}{\ensuremath{n}}

\newcommand{\totalNindice}{\ensuremath{n}}

\newcommand{\totalMindice}{\ensuremath{m}}

\newcommand{\fol}{\mathcal{F}}

\def\mc{\mathcal}

\def\fol{\mc{F}}

\begin{document}



\title[MCF of SRF: non compact case]{On Mean curvature flow of Singular Riemannian foliations: Non compact cases}


\author[Alexandrino]{Marcos M. Alexandrino}

\author[Cavenaghi]{Leonardo F. Cavenaghi}

\author[Gon\c{c}alves]{Icaro Gon\c{c}alves}


\address[Alexandrino]{
Universidade de S\~{a}o Paulo, 
Instituto de Matem\'{a}tica e Estat\'{\i}stica,
 Rua do Mat\~{a}o 1010,05508 090 S\~{a}o Paulo, Brazil}
\email{marcosmalex@yahoo.de, m.alexandrino@usp.br}
\address[Cavenaghi]{
Universidade de S\~{a}o Paulo, 
Instituto de Matem\'{a}tica e Estat\'{\i}stica,
 Rua do Mat\~{a}o 1010,05508 090 S\~{a}o Paulo, Brazil}
\email{leonardofcavenaghi@gmail.com, kvenagui@ime.usp.br}
\address[Gon\c{c}alves]{Centro de Matem\'atica, Computa\c{c}\~ao e Cogni\c{c}\~ao,
Universidade Federal do ABC, 09.210-170, Santo Andr\'e, Brazil.
}
\email{icaro.goncalves@ufabc.edu.br}

\thanks{Marcos M. Alexandrino was  supported by grant  $\#$2016/23746-6, S\~{a}o 
Paulo Research Foundation (FAPESP).
Leonardo Cavenaghi was  supported by grant  $\#$2017/24680-1, S\~{a}o 
Paulo Research Foundation (FAPESP) and acknowledges his PhD scholarship.
 \'{I}caro Gon\c{c}alves was supported in part by the Coordena\c{c}\~ao de Aperfei\c{c}oamento de
Pessoal de N\'ivel Superior -- Brasil (CAPES) -- Finance Code 001.}

\date{\today}


\subjclass[2000]{}


\begin{abstract}
In this paper we investigate the  
mean curvature flow (MCF) of a regular leaf of a closed generalized isoparametric foliation 
 as initial datum, generalizing previous results of Radeschi and the first author. 
We show that, under bounded curvature conditions,  any finite time singularity
is a singular leaf, and the singularity is of type I. 
We also discuss the existence of  basins of attraction, 
how cylinder structures can affect convergence of basic MCF of immersed submanifolds 
and  make a few remarks  on MCF of non-closed leaves of generalized isoparametric foliation.

 \end{abstract}

\maketitle





\section{Introduction}

A singular foliation $\F$ on a complete Riemannian manifold $M$ is called \emph{singular Riemannian foliation} (SRF)  if 
every geodesic  perpendicular  to one leaf   is perpendicular to every leaf it meets, see \cite[page 189]{Molino}.
Recall that a leaf of a singular Riemannian foliation is called regular if it has
maximal dimension, and singular otherwise.
In addition,  if the mean curvature vector field along regular leaves is basic, the foliation is called
\emph{generalized isoparametric foliation.}

A typical example of a generalized isoparametric foliation  is  the partition of a Riemannian manifold into the connected components of the orbits of an isometric action (the homogenous examples). Other classical examples are the  families of isoparametric foliations on 
Euclidean and symmetric spaces. In addition, all examples of SRF with closed leaves  on Euclidean or round sphere are \emph{generalized isoparametric}, and 
there are infinitely many  nonhomogeneous examples in these spaces, see \cite{Radeschi-Clifford}.  
For  more detailed information on  generalized isoparametric foliations see
Sections 1 and 2 of \cite{AlexandrinoRadeschi-mean-curvature}.

In  \cite{AlexandrinoRadeschi-mean-curvature}, Radeschi and the first author 
 studied the mean curvature flow of a regular leaf of a generalized isoparametric foliation $\F$ as initial datum
assuming that the ambient space $M$ is compact as well the leaves of $\F$. They proved that 
   any finite time singularity is a singular leaf, and the singularity is of type I, generalizing results of Liu--Terng \cite{Liu-Terng} and Koike \cite{Koike}.

Recall that   a smooth  family of immersions $\varphi_t: L_{0}\to M$,
$t\in[0,T)$  is called a solution of the 
\emph{mean curvature flow} (MCF for short) if $\varphi_t$ satisfies the evolution equation
$$\frac{d}{d t} \varphi_{t}(x)= H(t,x),\qquad  $$
where $H(t,x)$ is the mean curvature of $L(t):=\varphi_{t}(L_{0})$ at $x$. 
We say that the MCF $\varphi_t$ has \emph{initial datum} $L_{0}$. 
By abuse of notation, we will often identify $\varphi_t$ with its image $L(t)$, and we will talk about the MCF flow $L(t)$.
For more details on MCF, see e.g., \cite{survey-Coldingatal}.

In this paper we generalize \cite{AlexandrinoRadeschi-mean-curvature},
 dropping the condition of compactness of $L$ and $M,$ replacing it with other weaker conditions.

\begin{theorem}
\label{theorem-main1}
Let $(M,\fol)$ be a generalized isoparametric foliation with closed  leaves on a complete manifold $M$ so that  
 $M/\F$ is compact. Let $L_0\in \fol$ be a regular leaf of $M$ and let $L(t)$ denote the mean curvature flow evolution of $L_0$ with maximal interval of existence $[0,T)$. Assume that $T<\infty$.  Then the following statements hold:
\begin{enumerate}
\item[(a)]  $L(t)$ converges (in the leaf space sense) to a singular leaf $L_T$ of $\F.$
\item[(b)]  If the curvature of $M$ is bounded and the shape operator along each leaf is bounded, then 
for each $p\in L(0)$  the line integral of MCF 
  $\varphi_{t}(p)$ converges to a point of $L_T$.  In addition 
the singularity is of type I, i.e.,
\[
\limsup_{t\to T^-}\|A_t \|^2(T-t)<\infty,
\]
where $\|A_t \|$ is the sup norm of the second fundamental form of $L(t)$.
\end{enumerate}
\end{theorem}
\begin{remark}
Since  the leaves of a SRF are  locally equidistant (recall \cite{AlexToeben2}),  
item (a) of the above theorem implies that $L(t)$ converges to a singular leaf $L_T$ in the Gromov-Haudsdorff sense
(recall \cite[Chapter 10]{petersen}). In addition, under bounded curvature conditions (i.e., $M$ has bounded curvature and the
shape operator of each leaf of $\F$ on $M$ is bounded) Lemma \ref{bounded-Tau}  implies  that, for each $\epsilon,$
we can find a small $r_0$ so that the metric projection $\rho:\tub_{r_{0}}(L_T)\to L_T$ 
restricted to $L(t)\subset \tub_{r_{0}}(L_T) $ turns to be an $\epsilon$-isometry i.e., 
$$ | d(x,y)-d(\rho(x),\rho(y))|<\epsilon $$ for each $x,y\in L(t).$ 
\end{remark}
One of the  key observations behind the proof of Theorem  \ref{theorem-main1} is the following useful fact; see Lemma \ref{lemma1-semilocalstructure}.
\begin{lemma}
\label{lemma-introducao-lemma1-semilocalstructure}
 For each SRF  $\F$ with closed leaves on a complete manifold $M$, and
a singular leaf $L_q,$  one can find for a tubular neighborhood $U$ of $L_{q}$ and a (Sasaki) metric $\metric\supind{0}$
so that the restricted foliation
$\F|_{U}$ turns to be a generalized isoparametric foliation on $(U,\metric\supind{0})$, where the principal curvatures
associated to each basic vector field along each regular leaf are constant.   
\end{lemma}
\begin{remark}
From the proof of the above lemma, one can check that 
 the holonomy foliation (with compact holonomy) restricted to the unit  bundle and 
others more elaborate examples
presented in \cite{Alexandrino-Inagaki-Struchiner} 
fulfill the hypothesis of Theorem \ref{theorem-main1}.
\end{remark}

As in \cite{AlexandrinoRadeschi-mean-curvature}, 
the main idea of the proof of item (a) of Theorem  \ref{theorem-main1} is to
assure the existence of \emph{basins of attraction}. More precisely, 
we have that \emph{for each singular leaf $L_{q}\in \F$ there exists a small tubular neighborhood $\Tub_{\epsilon}(L_{q})$ 
so that for each regular leaf $L(t_{0})$ contained in  $\Tub_{\epsilon}(L_{q})$, the MCF $t\to L(t)$
continues to stay in $\Tub_{\epsilon}(L_{q})$ for $t>t_{0}$. In addition if $L(t_{0})\subset\Tub_{\epsilon}(L_{q})$ we have  $T< \infty$;} see Lemma \ref{newlemma_r_t} for details.   

As we  remark in Section \ref{sectionMCFNoclosed}, Lemma \ref{newlemma_r_t}   can be adapted to the case of SRF 
with non-closed leaves. As a simple application, \emph{we can assure convergence of MCF of $t\to L(t)$ when $T<\infty$ and $M$ is compact,}
see Proposition \ref{proposition-SRFnoclosed}.

Under bounded curvature conditions, another  adaption of  Lemma \ref{newlemma_r_t} 
can also be  useful to prove \emph{the convergence of MCF of an  immersed  submanifold $N$
contained in a regular leaf  as initial datum, when  the MCF of $N$ 
 can be extented to a basic flow of $\F$},  see  Proposition \ref{theorem-main3}. 
This adaption of Lemma \ref{newlemma_r_t}, and hence the proof of Proposition \ref{theorem-main3}, will follow direct from  
 the estimate in Lemma \ref{prop:submanifold}, an interesting remark of  immersion theory
 that we could not find in the classical literature. 
It states that \emph{given a Riemannian manifold $(M,\metric)$ with bounded curvature 
and an  embedded submanifold $L$ with bounded shape operator,
we can have a control of the trace of the shape operator $A_{\nabla r}$ of 
a immersed submanifold $N\subset \partial \mathrm{Tub}_{\epsilon}(L)$  with respect to the gradient of the distance $r$ to $L$,
as long as, we have a well defined tubular neighborhood 
$\mathrm{Tub}_{\epsilon}(L)$ of $L,$ for a small $\epsilon$.}

This	 paper is divided as follows. 
In Section \ref{section-sasakimetric} we prove Lemma \ref{lemma-introducao-lemma1-semilocalstructure} 
that will be important in the proof of  Theorem \ref{theorem-main1} presented in Sections    
\ref{sectionProofitem-a-of-TheoremA} and \ref{section-Proof-ofitem-B-bounded}. 
In Section \ref{sectionMCFNoclosed} we remark a few results on MCF of SRF with non-closed leaves, and  
Proposition \ref{proposition-SRFnoclosed} is presented.  
Finally, in Section \ref{section-onvergenceMCFviacomparisonlemma} 
we prove Proposition \ref{theorem-main3} via  the estimate in Lemma \ref{prop:submanifold}.
In the appendix we also present  
 Lemma \ref{bounded-Tau}   concerning the behavior of the distribution $\mathcal{T}$
used to define the Sasaki metric in Lemma \ref{lemma-introducao-lemma1-semilocalstructure}.
This lemma play a role in the proof of item (b) of Theorem \ref{theorem-main1} and may be relevant
in  future studies of SRF under bounded curvature conditions. 


\section{The distribution $\mathcal{T}$ and the Sasaki metric }
\label{section-sasakimetric}

As discussed in \cite{Alexandrino-Inagaki-Struchiner,Alexandrino-Radeschi-Molino},  
given a closed leaf $L_q$ we can find a 
$\F$-saturated tubular neighborhood  $U=\tub_{\epsilon}(L_{q})$ of $L_q$
and a subfoliation $\F^{\ell}\subset \F|_{U}$
(the \emph{linearized foliation}) that is 
the maximal infinitesimal homogenous subfoliation
of $\F$. In other words, if $\rho:U\to L_q$ is the
metric projection, and $S_p=\rho^{-1}(p)$ is the
slice (i.e., $S_{p}:=\exp_{p}(\nu_{p}L\cap B_{\epsilon}(0))$), then $\F_{p}^{\ell}=S_{p}\cap \F^{\ell}$
is the maximal homogenous subfoliation of 
the infinitesimal foliation $\F_{p}=S_{p}\cap \F.$ 
The infinitesimal foliation $\F_p$
turns to be a SRF on the Euclidean space
$(S_p,g_p)$ if we identify $S_p$ 
via the exponential map with an open set of $\nu_p(L_q)$
with the flat metric $g_p$. 
In addition, we can find a distribution  $\mathcal{T}$ homothetic invariant (with respect to
$\rho$) that is tangent to $\F^{\ell}$ and extends $T(L_q).$

Set  $U\supind{0}:=\exp^{-1}(U).$ The distribution and the foliation on $U\supind{0}$
defined by the pullback of $\mathcal{T}$ and $\F^{\ell}$ through 
the normal exponential map will also be denoted by $\mathcal{T}$ and $\F^{\ell}.$
Let $\metric\supind{0}$  be the \emph{associated Sasaki metric}, i.e., 
the metric so  that $\mathcal{T}$ is orthogonal to $\nu(L_q),$ the foot point projection   
 $\rho\supind{0}:(\nu(L_q),\metric\supind{0})\to (L_q,g)$ is a Riemannian submersion and
the fibers $\nu_{p}(L_{q})$ have the flat metric $\metric_{p}$. 
The foliation $\F\supind{0}:=(\exp^{\nu})^{-1}(\F|U)$ turns to be a SRF with respect to $\metric\supind{0}$
on $U\supind{0}$ and by homothetic transformation it can be extended to $\nu(L_{q})$.  
Let us denote $\nabla\supind{0}$ the Riemannian connection associated to $\metric\supind{0}.$

Now we present a useful  application of the above discussion. 

\begin{lemma}
\label{lemma1-semilocalstructure}
Let $\F$ be a SRF with closed leaves and $L_{q}$ be a singular leaf. Consider the 
SRF $\F\supind{0}$ on  $(\nu(L_{q}), \metric\supind{0})$ defined above. Then
\begin{enumerate}
\item[(a)] $A_{\xi}\supind{0}|_{\mathcal{T}}=0,$ where $\xi$ is a normal vector field along a regular leaf  $L_x$. 
\item[(b)] The principal curvatures associated to basic vector fields along regular leaves of $\F\supind{0}$ 
are constant. In particular $\F\supind{0}$ is  a generalized isoparametric foliation. 
\item[(c)] The principal directions associated to non zero curvatures are tangents to the fibers of $\nu(L_q).$ 
\item[(d)] $\nabla\supind{0}\xi |_{\mathcal{T}}=0,$ if $\xi$ is the gradient of the distance function 
 $r(x)=d^{0}(L_{q},x)$ i.e.,  the distance between $x$ and $L_q$ with respect to the metric $\metric\supind{0}$.
\end{enumerate} 
\end{lemma}
\begin{proof}


Once $\metric\supind{0}$ is a Sasaki metric, 
 the fibers of $\nu(L_{q})$ are totally geodesics and  isometric to each other. Therefore 
the space $T_{x}\big(\nu_{\rho(x)}(L_{q})\cap L_{x})$ is $A_{\xi}\supind{0}$-invariant and 
hence the distribution $\mathcal{T}_{x}$ is also $A_{\xi}\supind{0}$-invariant. Also recall that
\emph{the principal curvatures associated to basic vector fields along regular leaves of
the infinitesimal foliation $\F\supind{0}\cap \nu_{\rho(x)}(L_{q})$  
are constant}; see \cite[Remark 3.2]{AlexandrinoRadeschi-I}. These facts together imply that items (b) and (c) of the lemma will be proved once we have checked 
item (a).

Given vector fields $X_1$ and $X_2$ on $L_q$, consider their lifts $X_{i}^{\tau}$ tangent to $\mathcal{T}$ and
$\xi$ a normal vector field along $L_x$. 
Let us also denote $\nabla^{b}$ the induced Riemannian  connection on $L_q$. 
As $\rho\supind{0}$ is a Riemannian submersion we have
\begin{equation}
\label{eq1-lemma1-semilocalstructure}
\nabla_{X_{1}^{\tau}}\supind{0}X_{2}^{\tau}= \big( \nabla_{X_1}^{b}X_{2} \big)^{\tau}+\frac{1}{2}[X_{1}^{\tau},X_{2}^{\tau}]^{v}. 
\end{equation}
Since $\mathcal{T}$ is tangent to $L_x$ and $\xi$ is orthogonal to $L_x$ we infer that:
\begin{equation}
\label{eq2-lemma1-semilocalstructure}
\metric\supind{0}(\xi,\big( \nabla_{X_1}^{b}X_{2} \big)^{\tau})=0.
\end{equation}
Since the (possible nointegrable) distribution $\mathcal{T}$ is tangent to $\F^{\ell}\subset\F,$ we have that  
$[X_{1}^{\tau},X_{2}^{\tau}]$ is tangent to $\F^{\ell}$ and hence:
\begin{equation}
\label{eq3-lemma1-semilocalstructure}
\metric\supind{0}(\xi,\frac{1}{2}[X_{1}^{\tau},X_{2}^{\tau}]^{v})=0.
\end{equation}
From Eq. \eqref{eq1-lemma1-semilocalstructure} \eqref{eq2-lemma1-semilocalstructure} \eqref{eq3-lemma1-semilocalstructure}
we conclude that
\begin{equation}
\label{eq4-lemma1-semilocalstructure}
\metric\supind{0}(A_{\xi}\supind{0}(X_{1}^{\tau}), X_{2}^{\tau})= \metric\supind{0}(\xi,\nabla_{X_{1}^{\tau}}\supind{0}X_{2}^{\tau})=0.
\end{equation}
This  finishes the proof of item (a) and hence the proof of items (b) and (c), as discussed above.  
 
In order to prove item (d), consider the geodesic variations $f(s,t)=\exp_{\beta(t)}(s\xi(t))$ so that
\begin{equation}
\label{eq-1-item-d-lemma}
\frac{\partial f}{\partial t}(s,0)=J(s)\in \mathcal{T}; 
\end{equation}
here $t\to \beta(t)$ is a curve in $L_{q}$ and 
$t\to \xi(t)$ is a unit normal field along $L_q$. Set $\gamma(s):=f(s,0).$ 
Since $\mathcal{T}$ is orthogonal to the totally geodesic submanifold $\nu_{q}(L),$ 
we conclude from Eq. \eqref{eq-1-item-d-lemma} that $J'(s)\in\mathcal{T}_{\gamma(s)}$. 
Therefore 
\begin{equation}
\label{eq-2-item-d-lemma}
\nabla\supind{0}_{\frac{\partial f}{\partial t}(s,0)}\xi=\frac{\nabla\supind{0}}{\partial t}\frac{\partial f}{\partial s}(s,0) =\frac{\nabla\supind{0}}{\partial s}\frac{\partial f}{\partial t}(s,0)= J'(s)\in\mathcal{T}_{\gamma(s)}
\end{equation} 
Equation \eqref{eq-2-item-d-lemma} and item (a) imply item (d) of the lemma. 

\end{proof}


\section{Proof of item (a) of Theorem \ref{theorem-main1}}
\label{sectionProofitem-a-of-TheoremA}

\subsection{A new  estimate of the shape operator}
\label{subsection-shapeoperatorestimative}

In this section we  generalize   a estimate  of 
\cite{AlexandrinoRadeschi-mean-curvature} that will 
 allow us to prove item (a) of Theorem \ref{theorem-main1}.

Let us start by  fixing some notations that will be used in the proof of Lemma \ref{lemma-constant-R-C}.
Given the original metric $\metric$ on $M$,  the metric on 
a  neighborhood $U\supind{0}:=\exp^{-1}(\Tub_{\epsilon}(L_{q}))$  of the null section of $\nu(L_{q})$ defined by the pullback of $\metric$ 
via  
the normal exponential map will also be denoted by $\metric$.   
Let
$\nabla$ be the Riemannian connection associated to $\metric$ on $U\supind{0}$. 
 Consider  the connection $\nabla\supind{0}$ associated to $\metric\supind{0}$  and set
$\omega:=\nabla- \nabla\supind{0}$.     
Consider an orthonormal basis $\{ e_{\totalFindice}\}$  of $T_{x}L_{x}$ with respect to the original metric $\metric$ so that
$e_{\sliceindice}\in T_{x}(\nu_{\rho(x)}(L_{q})\cap L_{x})$ (for $\sliceindice=1\cdots k$) and 
$e_{\normalindice}\in T_{x}L_{x}$  (for $\normalindice=k+1\cdots \dim \F$). 

\begin{lemma}
\label{lemma-constant-R-C}
Let $L_q$ be a closed singular leaf. Then  
there exist  a  radius $\epsilon$ and constant $c_1>0$ 
such that in the tubular neighborhood $\tub_{\epsilon}(L_q)$ the following
equation holds true:
\begin{equation}\label{Lemma-bounds_S}
-\frac{k}{\DR(x)}-c_1\leq {\rm tr} (A_{\nabla \DR})_{x} \leq -\frac{k}{\DR(x)}+c_1,
\end{equation}
where $r(x)=d(L_{q},x)$ is the distance between the regular points $x\in L(t)\subset \tub_{\epsilon}(L_q)$ and the singular leaf $L_q$
and $k=\dim \F-\dim L_{q}$.  
\end{lemma}
\begin{proof}

Let $\xi$ be the gradient of the distance function $r$ with respect to 
$\metric$ (or with respect to $\metric\supind{0}$ that gives the same gradient for the function $r$).    

\begin{eqnarray*}
\mathrm{tr}A_{\xi} & = & - \sum_{\totalFindice=1}^{\dim \F} \metric(\nabla_{e_{\totalFindice}}\xi, e_{\totalFindice})\\
& = & -\sum_{\sliceindice=1}^{k}\metric (\nabla_{e_{\sliceindice}}\supind{0}\xi, e_{\sliceindice}) 
-\sum_{\normalindice=k+1}^{\dim \F}\metric (\nabla_{e_{\normalindice}}\supind{0}\xi, e_{\normalindice})\\
& - & \sum_{\totalFindice=1}^{\dim \F}\metric(\omega(e_{\totalFindice})\xi, e_{\totalFindice})
\end{eqnarray*}

Now let us examine each of the above terms. 
We know from Euclidean geometry that 
$$  -\sum_{\sliceindice=1}^{k}\metric (\nabla_{e_{\sliceindice}}\supind{0}\xi, e_{\sliceindice})=
-\sum_{\sliceindice=1}^{k} \metric(\frac{1}{r}e_{\sliceindice}, e_{\sliceindice})
= - \frac{k}{r}.$$ 
For $k+1\leq \normalindice\leq \dim \F$ set  $e_{\normalindice}=e_{\normalindice}^{\nu}+e_{\normalindice}^{\tau}$ where 
$e_{\normalindice}^{\nu}\in T_{x}(\nu_{\rho(x)}(L_{q})\cap L_{x})$ and $e_{\normalindice}^{\tau}\in\mathcal{T}.$ 
Also recall from  Lemma \ref{lemma1-semilocalstructure} that $\nabla\supind{0}\xi|_{\mathcal{T}}=0.$  
Then  
\begin{eqnarray*}
-\sum_{\normalindice=k+1}^{\dim \F}\metric (\nabla_{e_{\normalindice}}\supind{0}\xi, e_{\normalindice})&= & 
-\sum_{\normalindice=k+1}^{\dim \F}\metric (\nabla_{e_{\normalindice}^{\nu}}\supind{0}\xi, e_{\normalindice})-\sum_{\normalindice=k+1}^{\dim \F}\metric (\nabla_{e_{\normalindice}^{\tau}}\supind{0}\xi, e_{\normalindice})\\
 & = & - \frac{1}{r}\sum_{\normalindice=k+1}^{\dim \F}\metric (e_{\normalindice}^{\nu}, e_{\normalindice}) + 0=0,
\end{eqnarray*}
where the last equality follows from the fact that $\metric(X,e_{\normalindice})=0$
for each $X\in T_{x}(\nu_{\rho(x)}(L_{q})\cap L_{x}).$

From the equations above we infer that

\begin{equation}
\label{eq-1-lemma-constant-R-C}
\mathrm{tr}\, A_{\xi}= -\frac{k}{r} -\mathrm{tr}\, \omega(\cdot)\xi.
\end{equation}
Equation \eqref{eq-1-lemma-constant-R-C} implies that $\mathrm{tr}\, \omega(\cdot)\xi$ is basic. 
On the other hand, in a relative compact neighborhood of 
 $q$  
\begin{equation}
\label{eq-2-lemma-constant-R-C}
-c_{1}\leq -\mathrm{tr}\,\omega(\cdot)\xi\leq c_{1}.
\end{equation}

Now Eq. \eqref{Lemma-bounds_S} follows from 
 \eqref{eq-1-lemma-constant-R-C} and \eqref{eq-2-lemma-constant-R-C}.

\end{proof}


\subsection{Revised proof}
\label{subsectionproof-item (a)}

Once we have proved  Lemma \ref{lemma-constant-R-C}, 
the proof of item (a) of Theorem \ref{theorem-main1} follows from  
 the same arguments as in  \cite{AlexandrinoRadeschi-mean-curvature}. 
For the sake of completeness let us briefly recall
 these arguments extracted  from \cite{AlexandrinoRadeschi-mean-curvature}.

\begin{lemma}[Basins of attraction \cite{AlexandrinoRadeschi-mean-curvature}]
\label{newlemma_r_t}
Let $\F$ be a generalized isoparametric foliation with closed leaves   on a complete Riemannian manifold $M$. 
Let $L_q$ be a singular leaf. Then there exists  a neighborhood $\tub_{\epsilon}(L_{q})$ around $L_{q}$ with radius 
$\epsilon$ small enough such that if the initial data $L(t_{0})\subset \tub_{\epsilon}(L_{q})$ then the following properties hold true:
\begin{enumerate}
\item[(a)] Let $L(t)$ be the MCF with initial data $L(t_{0})$ and $r(t)$ the distance between $L(t)$ and $L_q$. Then 
\begin{equation}
\label{eq-0-lemma_r_t}
C_{1}^{2}(t-t_{0})\leq \DR^{2}(t_{0})-\DR^{2}(t)\leq C_{2}^{2}(t-t_{0}),
\end{equation} 
where $C_1$ and $C_2$ are positive constants that depend only on $\tub_{\epsilon}(L_{q})$.
\item[(b)] $T< \infty$ and $L(t)\subset \tub_{\epsilon}(L_{q})$ for all $t>t_0.$  
\item[(c)] If $L(t)$ converges to $L_q$ then
\begin{equation}
\label{eq-1-lemma_r_t}
C_1\sqrt{T- t}\leq \DR(t)\leq C_2\sqrt{T- t}.
\end{equation}

\end{enumerate}
\end{lemma}
\begin{proof}

We start with a small $\epsilon_0$  
so that the   distance function $\DR(x)=d(L_{q},x)$ with respect to a singular leaf $L_{q}$ is smooth on 
$\tub_{\epsilon}(L_{q})\setminus L_q .$
 Let $p\in L(t_0)$ and consider the solution of the MCF $\varphi$ with inicial condition $p$ i.e., the curve
$t\to \varphi_t(p)$ such that  $\frac{d}{d t}\varphi_t(p)=H(t)$. Then we have 
\begin{eqnarray*}
r'(t) &=& \frac{d}{d t} \DR\circ \varphi_{t}(p) \\
      &=& \langle \nabla \DR, \varphi'_{t}(p) \rangle \\ 
			&=& \langle \nabla \DR, H(t) \rangle \\
			&=& {\rm tr}(A_{\nabla \DR}).
\end{eqnarray*}

 From Lemma \ref{lemma-constant-R-C}  we  have:
\begin{equation*}
-\frac{k}{\DR}-c_{1}\frac{r}{r} \leq {\rm tr} A_{\nabla \DR} \leq -\frac{k}{\DR}+c_{1}\frac{r}{r}.
\end{equation*}

Now we chose $\epsilon < \min \{\epsilon_{0}, \frac{k}{c_1} \}$
and define the constants $C_1$, $C_2$ by the equations
\[
\left\{\begin{array}{l} \frac{C_1^2}{2}= k-\epsilon \cdot c_{1},\\ \\ 
\frac{C_2^2}{2}=k+ \epsilon \cdot c_{1}\end{array}\right.
\]
The above equations  imply
\begin{equation*}
-\frac{C_2^2}{ 2 \DR(t)} \leq \DR \,'(t) \leq -\frac{C_1^2}{2 \DR(t)}
\end{equation*}
 or, equivalently, $-C_2^2\leq (\DR^2(t))'\leq -C_1^2$. 
Integrating this equation we get
\begin{equation}
\label{eq-lemma_r_t-final}
C_{1}^{2}(t-t_{0})\leq \DR^{2}(t_{0})-\DR^{2}(t)\leq C_{2}^{2}(t-t_{0})
\end{equation} 
for $t>t_{0}$ closer to $t_{0}$ and hence for every $t> t_{0}$. This conclude the proof of item (a).
Itens (b), (c)  follow directly from item (a). 

\end{proof}

Let $\pi: M\to M/\F$ be the canonical projection.
Since $t\to \pi(L(t))$ is contained in a compact set and $T$
is finite, the limit set of $t\to \pi(L(t))$ cannot 
be contained in the regular stratum and thus
it must be contained in the singular one. 
Let $L_{q}$ be a leaf  in the limit set, 
and consider a sequence $t_{n}\subset [0,T)$
so that $t_n\to T$ and $\pi(L(t_{n}))\to \pi(L_{q})$. 
 Given small $\epsilon$, Lemma \ref{newlemma_r_t} implies that  
there exists $t_{n_{0}}$ so that
if $t>t_{n_{0}}$ then $L(t)\in\Tub_{\epsilon}(L_{q}).$
The arbitrariness of $\epsilon$ implies that 
$\pi(L(t))$ converges to $\pi(L_{q})$.

\section{Proof of item (b) of Theorem \ref{theorem-main1} }
\label{section-Proof-ofitem-B-bounded}


\subsection{New  estimate of the shape operator under bounded curvature conditions}

In this section we  generalize    estimates  in 
\cite{AlexandrinoRadeschi-mean-curvature} and these will 
 allow us to prove item (b) of Theorem \ref{theorem-main1}. We are going to use the same convention 
for local frame established in Section \ref{subsection-shapeoperatorestimative}.

\begin{lemma}
\label{lemma-estimative-H-boundedcurvature}
Let $L_q$ be a closed  singular leaf and assume that there exists a tubular neighborhood of $L_q$ with bounded curvature, i.e.,
$-k_1 \leq K\leq k_1$ for a positive constant $k_1.$ 
Then, reducing the tubular neighborhood if necessarily, there exists $t_{0}>0$ so that for $t_{0}<t<T$ we have: 
\begin{equation} 
\label{eq-mean-curvature-supShapeoperatorTau}
\| H(t)\|\leq C_{1}\| A\supind{0}(t)\|\subind{0} +C_2.
\end{equation}
\end{lemma}

\begin{proof}
Consider the $(1,1)$-tensor field $G$ defined as $\metric(X,Y)=\metric\supind{0}(G X, Y)$. 
For $y$ close to $L(t),$ 
let $r\supind{0}(x)=d\supind{0}(x,L_{y})$ be  the distance function  
with respect to the metric $\metric\supind{0}.$  A direct calculation implies that 
\begin{equation}
\label{eq1-eq-mean-curvature-supShapeoperatorTau}
\nabla r\supind{0}=G^{-1}\nabla\supind{0} r\supind{0}.  
\end{equation}

Using the fact that $r\supind{0}$ is a $\F$-basic function, 
and  $U\supind{0}$ is a saturation of a relative compact neighborhood of $q$, it is straightforward to check the following properties:

\begin{claim}
\label{claim1Marcos}
\

\begin{enumerate}
\item $\nabla r\supind{0}$ is basic; 
\item $\metric (\nabla r\supind{0},\nabla r\supind{0}) $ is constant along regular leaves;
\item $c_{1}< \sqrt{\metric (\nabla r\supind{0},\nabla r\supind{0})} < c_{2}$ on $U_0$, 
where $c_i$ is a constant that does not depend on $\nabla r\supind{0}$.  
\end{enumerate} 
\end{claim}

From Eq. \eqref{eq1-eq-mean-curvature-supShapeoperatorTau}  and  
 $(\nabla\supind{0}_{(\cdot)}G^{-1})=-G^{-1}\big(\nabla_{(\cdot)}\supind{0} G\big)G^{-1}$ 
we have:

\begin{equation}
\label{eq2-eq-mean-curvature-supShapeoperatorTau}
\nabla\supind{0}_{e_{\totalFindice}} \nabla r\supind{0}= G^{-1} \nabla_{e_{\totalFindice}}\supind{0} \nabla\supind{0} r\supind{0} - G^{-1}\big( \nabla_{e_{\totalFindice}}\supind{0}G \big)G^{-1} \nabla\supind{0} r\supind{0}.
\end{equation}
Since $\nabla_{e_{\totalFindice}}\nabla r\supind{0} =\nabla\supind{0}_{e_{\totalFindice}}\nabla r\supind{0}+\omega(e_{\totalFindice})\nabla r\supind{0}$ we have from
Eq.~\eqref{eq2-eq-mean-curvature-supShapeoperatorTau} that:

\begin{eqnarray*}
-  \metric(H(t),\nabla r\supind{0} ) & = & -\mathrm{tr} A_{\nabla r\supind{0}}  =  \sum_{\totalFindice} \metric(\nabla_{e_{\totalFindice}} \nabla r\supind{0}, e_{\totalFindice} )\\
& = &
\sum_{\totalFindice} \metric(\nabla\supind{0}_{e_{\totalFindice}}\nabla r\supind{0}, e_{\totalFindice})+\sum_{\totalFindice} 
\metric(\omega(e_{\totalFindice})\nabla r\supind{0},e_{\totalFindice})\\
&= & \sum_{\totalFindice}\metric\supind{0}( \nabla_{e_{\totalFindice}}\supind{0} \nabla\supind{0} r\supind{0},e_{\totalFindice}) 
- \sum_{\totalFindice}\metric(G^{-1}\big( \nabla_{e_{\totalFindice}}\supind{0}G \big)G^{-1} \nabla\supind{0} r\supind{0},e_{\totalFindice})\\
& + & \sum_{\totalFindice} \metric(\omega(e_{\totalFindice})\nabla r\supind{0},e_{\totalFindice}).
\end{eqnarray*}
In what follows we are going to prove that there exists $c_3$  so that
\begin{equation}
\label{eq3-eq-mean-curvature-supShapeoperatorTau}
|\metric\supind{0}( \nabla_{e_{\totalFindice}}\supind{0} \nabla\supind{0} r\supind{0},e_{\totalFindice})|< c_{3} \| A\supind{0}(t)\|\subind{0}
\end{equation}
Set $\mathcal{B}:=-\sum_{\totalFindice}\metric(G^{-1}\big( \nabla_{e_{\totalFindice}}\supind{0}G \big)G^{-1} \nabla\supind{0} 
r\supind{0},e_{\totalFindice})
 +  \sum_{\totalFindice} \metric(\omega(e_{\totalFindice})\nabla r\supind{0},e_{\totalFindice})$.
Note that $\mathcal{B}$ is well defined along the regular leaves (its definition does not depend on the frame $\{e_{\totalFindice}\}$).  
The fact that the mean curvature and $\| A\supind{0}(t)\|\subind{0}$  are basic and 
Eq.\eqref{eq3-eq-mean-curvature-supShapeoperatorTau}
will imply that $\mathcal{B}$ is bounded along each regular leaf and hence 
(since is bounded on relative compact  a neighborhood of $q$) bounded on the regular stratum of $U\supind{0}$ 
i.e., $|\mathcal{B}|<c_{4}.$ These equations will then imply that
\begin{equation}
\label{eqintermediariaeq-mean-curvature-supShapeoperatorTau} 
 |\metric(H(t),\nabla r\supind{0} )|\leq c_{3} \| A\supind{0}(t)\|\subind{0}+  c_{4}.
\end{equation} 
The arbitrariness of $r\supind{0}$,
Eq. \eqref{eqintermediariaeq-mean-curvature-supShapeoperatorTau} and 
item (c) of Claim \ref{claim1Marcos}  allow us  to  infer  Eq.\eqref{eq-mean-curvature-supShapeoperatorTau}.  
Let us now prove Eq. \eqref{eq3-eq-mean-curvature-supShapeoperatorTau}.  

As in the previous lemma, we  denote $X^{\nu}$ 
the $\metric\supind{0}$-projection of a vector $X$ onto 
the fibers of $\nu(L_{q})$. 
By using Lemma \ref{lemma1-semilocalstructure}, we can check:
\begin{equation}
\label{eq5-eq-mean-curvature-supShapeoperatorTau}
\metric\supind{0}(\nabla\supind{0}_{e_{\normalindice}}\nabla\supind{0}r\supind{0},e_{\normalindice})=
\metric\supind{0}(\nabla\supind{0}_{e_{\normalindice}^{\nu}}\nabla\supind{0}r\supind{0},e_{\normalindice}^{\nu}).
\end{equation}

 Writing  
$X^{\nu}=\sum_{\beta} \metric\supind{0}(X,e\supind{0}_{\beta}) e\supind{0}_{\beta}$,
where $\{ e\supind{0}_{\beta} \}$
is a $\metric\supind{0}$-orthonormal  basis
of principal directions of 
$A\supind{0}_{\nabla\supind{0}r\supind{0}}$, it is easy to verify  the next equation:
\begin{equation}
\label{eq4-eq-mean-curvature-supShapeoperatorTau}
|\metric\supind{0}(\nabla\supind{0}_{X^{\nu}}\nabla\supind{0}r\supind{0},X^{\nu})|\leq  \|A^{\scriptscriptstyle 0}_{\nabla^{ \scriptscriptstyle 0} r^{\scriptscriptstyle  0}}\|_{ \scriptscriptstyle 0} \metric^{\scriptscriptstyle  0}(X^{\nu},X^{\nu}).
\end{equation}


\begin{claim}
\label{claim2Marcos}
There exist  constants $c_5$, $c_6$ (that depends only on radius $\epsilon$ and 
the bounded curvature) so that $c_{6} \, \metric(X^{\nu},X^{\nu}) \leq\metric\supind{0}(X^{\nu},X^{\nu})\leq c_{5} \, \metric(X^{\nu},X^{\nu})$, for every 
$X^{\nu}\in T_{x}\big(\nu_{\rho(x)}(L_q)\big)$.   
\end{claim}

In fact, consider $W\in \nu_{\rho(x)}(L_q)$ so that $\metric(rW,rW)=\metric\supind{0}(rW,rW)_{\rho(x)}=1$ and 
$\frac{X^{\nu}}{\| X^{\nu} \|}=\frac{J(r)}{\|J(r)\|}$ where  
 $J(s)= d (\exp_{\rho(x)})_{sv}(sW)$ is the associated Jacobi field.  
Since  $\exp\supind{0}_{\rho(x)}=\exp_{\rho(x)},$ we have that $J(s)=J\supind{0}(s)$ and hence 
$$\metric\supind{0}\Big(\frac{X^{\nu}}{\| X^{\nu}\|},\frac{X^{\nu}}{\| X^{\nu}\|}\Big)=\metric\supind{0}\Big(\frac{J(r)}{\| J(r) \|},\frac{J(r)}{\| J(r) \|}\Big)= 
\metric\supind{0}\Big(\frac{J\supind{0}(r)}{\| J(r)\|},\frac{J\supind{0}(r)}{\| J(r) \|}\Big) =\frac{1}{\| J(r)\|^{2}}.$$
Claim \ref{claim2Marcos} follows from  Rauch's theorem \cite[Chapter 10]{Manfredo}, that assures $\frac{1}{\sqrt{c_{5}}}\leq \|J(r)\|\leq \frac{1}{\sqrt{c_6}}$. 

\

From Lemmas \ref{lemma-parallelandslices} and \ref{bounded-Tau} and  Claim \ref{claim2Marcos}  we know that if $\metric(e_{\normalindice},e_{\normalindice})=1$ then 
$\metric^{\scriptscriptstyle  0}(e^{\nu}_{\normalindice},e^{\nu}_{\normalindice})$ is bounded. This fact, 
Eq. \eqref{eq5-eq-mean-curvature-supShapeoperatorTau}, \eqref{eq4-eq-mean-curvature-supShapeoperatorTau}
and Claim \ref{claim2Marcos} imply Eq. \eqref{eq3-eq-mean-curvature-supShapeoperatorTau},
which concludes the proof, as discussed before. We stress that the condition of bounded shape operators
along the leaves has been used in Lemma \ref{bounded-Tau}.
\end{proof}


\begin{remark}
The proof of Lemma \ref{lemma-estimative-H-boundedcurvature} also  implies that, for each $\widetilde{q}$ in the
 singular leaf $L_q$, there exists a (relative compact)  neighborhood $V$ of the point $\widetilde{q}$ 
and constants $C_1$ and $C_3$ so that
\begin{equation}
\label{eqremark-estimative-A-boundedcurvature}
\| A_{x}(t)\|\leq C_{1}\| A_{x}\supind{0}(t)\|\subind{0} +C_3.
\end{equation} 
for $x\in V$. Here, it is important to stress that the constant $C_3$ may depend on the neighborhood $V$ of $\widetilde{q}$ and 
may not be defined on tubular neighboorhood of $L_q$ in the case where $L_{q}$ is not compact. 
\end{remark}

\subsection{Revised proof}
\label{revisedproof-typeI}

Once we have proved Lemma  \ref{lemma-estimative-H-boundedcurvature}
and Eq.~\eqref{eqremark-estimative-A-boundedcurvature}, 
the proof of Item (b) of Theorem \ref{theorem-main1}
follows from a small adaptation of \cite{AlexandrinoRadeschi-mean-curvature}
as we now review. 


Given  a tubular neighborhood $\tub_{\epsilon}(L_{q}),$  we 
define $r_{\Sigma}: \tub_{\epsilon}(L_{q})\to \mathbb{R}$ 
as the distance between $L_x$ and the singular strata $\Sigma$, 
and $f:\tub_{\epsilon}(L_{q})\to \mathbb{R}$ as 
the distance between $L_x$ and its focal set. By abuse
of notation, we set $r_{\Sigma}(t):=r_{\Sigma}(L(t))$
and $f(t)=f(L(t))$. As proved in  \cite[Proposition 3.6]{AlexandrinoRadeschi-mean-curvature}  
$r_{\Sigma}(t)\geq C r(t)$ for 
$r(t)=d(L(t),L_{q})$. From 
Lemma \ref{newlemma_r_t} we can infer
that 
$r_{\Sigma}(t)\geq C \sqrt{T-t}.$  
As proved in \cite[Proposition 3.7]{AlexandrinoRadeschi-mean-curvature},  
there exists a constant $\sigma\in (0,1)$
such that $f(p)\geq \sigma r_{\Sigma}(p)$
for every regular point $p\in M$. 
These results hold on a tubular  neighborhood of $L_q$ 
for the original metric $\metric$ and metric  $\metric\supind{0}.$ Putting these results together 
we get
\begin{equation}
\label{eq-1-proofItemb}
f\supind{0}(t)\geq C\sqrt{T-t}, 
\end{equation}
where $f\supind{0}(t)$ is the distance between $L_x$ and its focal set with respect with
$\metric\supind{0}.$ 
On the other hand, by Lemma \ref{lemma1-semilocalstructure} we infer 
\begin{equation}
\label{eq-2-proofItemb}
\| A\supind{0}(t)\|\subind{0}=\frac{1}{f\supind{0}(t)}.
\end{equation}
Combining Eq.  \eqref{eq-1-proofItemb} and \eqref{eq-2-proofItemb} we have that 
\begin{equation}
\label{eq-3-proofItemb}
\| A\supind{0}(t)\|\subind{0}\sqrt{T-t}\leq C_3
\end{equation}
holds on $\tub_{\epsilon}(L_{q})$. Eq. \eqref{eq-3-proofItemb} and  Lemma  \ref{lemma-estimative-H-boundedcurvature} imply 
\begin{equation}
\label{eq-4-proofItemb}
\|H(t)\|\sqrt{T-t} \leq  C_{4}.  
\end{equation}

Let $\gamma(t):=\varphi_{t}(p)$ be the integral curve of $H$ starting at $p.$
Define $h:[0,1)\to[0,T)$ as $h(s):=T-T(1-s)^{2}$ and set $\beta(s):=\gamma(h(s)).$   
As consequence of Eq. \eqref{eq-4-proofItemb} we have 
 $\|\beta'(s)\|< \infty.$  
Since $L(t)$ converges to $L_q$ (in the leaf space sence) and $\beta$ has finite length,  
we conclude that $\beta$ converges to a point $q$, i.e., $\gamma$ converges to a point $q.$  
This fact, Eq. \eqref{eq-3-proofItemb} and  \eqref{eqremark-estimative-A-boundedcurvature}
imply that 
\begin{equation}
\label{eq-5-proofItemb}
\| A_{x}(t)\|\sqrt{T-t}\leq C_{5}, 
\end{equation}
for $x$ close to $q.$  Equation \eqref{eq-5-proofItemb} implies that the convergence is type I. 

\section{Remarks on MCF of non-closed regular leaf.  }
\label{sectionMCFNoclosed}

As proved in  \cite{Alexandrino-Radeschi-Molino},  if   $\F=\{ L \}$ is a SRF then  $\overline{\F}=\{ \overline{L}, L\in \F\}$ (i.e, 
partition of $M$ into the closures of the leaves of $\F$) is also a 
SRF. This was the so called \emph{Molino's conjecture}.

Note that mean curvature of $\overline{L_{q}}$  does not necessary coincide with the mean curvature of $L_{q}$ 
and hence  it would make sense to ask if we can say something 
about the MCF of a regular (non-closed) leaf as initial datum.

As we are going to explain,  
a part from a  small generalization 
in the semi-local model presented in Section \ref{section-sasakimetric}
(see  \cite{Alexandrino-Radeschi-Molino} and \cite{Alexandrino-Inagaki-Struchiner}),  
the proofs of  Lemmas \ref{Lemma-bounds_S}
and   \ref{newlemma_r_t} also hold for SRF with non-closed leaves. In this case 
the singular leaf $L_{q}$ can be replaced by its closure $\overline{L_{q}}$.

Let $B$ be a closed submanifold of $M$ saturated by  leaves of $\F$ with the same dimension  
(e.g, $B=\overline{L}$ or $B$ is the minimal stratum). 
Consider $U:=\tub_{\epsilon}(B)$ the tubular neighborhood 
and $\rho:U\to B$ the metric projection. Again, via normal exponential map, we can
identify $U$ with a neighborhood  of the null section $B,$
 $\F|_{U}$ with a foliation on the normal bundle $\nu(B)$ of $B,$
and the map $\rho$ with the foot point map $\nu(B)\to B.$  

There exists 3 homothetic distributions $(\mathcal{K},\mathcal{T},\mathcal{N})$ so that: 
\begin{itemize} 
\item $\mathcal{K}=\ker\rho_{*}$,  
\item $\mathcal{T}$ extends $T\F|_{B}$ and is everywhere
tangent to the leaves of $\F,$
\item $\mathcal{N}$ extends the normal space of  $T\F|_{B},$
\item $TM=\mathcal{T}\oplus\mathcal{N}\oplus\mathcal{K}.$
\end{itemize}
Let $\metric\supind{0}$ be the ``Sasaki" metric on the fiber bundle $\mathcal{K}\to B$ with respect to the
homothetic  distribution $\mathcal{T}\oplus\mathcal{N}.$ Note that in this case the fibers  $\nu(B)$ are not 
totally geodesic, because the product is 
compatible with the metric only if we derive in the direction of $\mathcal{T}$. 
The same proof of Lemma \ref{lemma1-semilocalstructure}   
allow us to conclude that 
\begin{equation}
\label{eq-1-sectionMCFNoclosed}
\metric\supind{0}(\nabla\supind{0}\xi|_{\mathcal{T}},X)=0  \mathrm{ \, \, \,  \forall \, \, }  X\in T_{x}L_{x},
\end{equation}
where $\xi=\nabla\supind{0} r,$ for $r(x)=d\supind{0}(\overline{L_{q}},x).$
Equation \eqref{eq-1-sectionMCFNoclosed}  allow us to infer the analogous to Lemma \ref{Lemma-bounds_S}
and  to Lemma \ref{newlemma_r_t}, 
if we replace the singular leaf $L_{q}$ with its closure $\overline{L_{q}}$.

In the particular case where $M$ is compact, we can use this adaptation  of Lemma \ref{newlemma_r_t} 
 to infer the convergence of MCF.

\begin{proposition}
\label{proposition-SRFnoclosed}
Let $M$ be a compact Riemannian manifold and $\F$ be a generalized isoparametric foliation on $M$, with possible non-closed leaves. 
Assume that the MCF $t\to L(t)$  of a regular leaf $L(0)$ as initial datum has $T<\infty$. Then $t\to L(t)$ must converge to 
the closure of a singular leaf. In addition, if for each $x\in M$ all leaves of the  
infinitesimal foliation $\F_x$ 
are compact (i.e., if $\F$ is 
infinitesimal compact) then 
for each $p\in L(0)$ the line integral of MCF $\varphi_{t}(p)$ 
converges  to a point of $L_{T}$ and  singularity  is 
of type I. 
    \end{proposition}
\begin{proof}

Since $M$ is a compact Riemannian manifold, we know that the singular strata $\Sigma$ (i.e., the union  of singular leaves) is also compact,
because  it is closed in $M$. Therefore, one can cover $\Sigma$ with a finite union of small tubular neighborhoods  
$\{\Tub_{\epsilon}(\overline{L_{q_i}})\}$ (the basins of attraction). This property, the fact that the mean curvature is bounded on the precompact set $M\setminus\cup \Tub_{\epsilon}(\overline{L_{q_i}})$ and the the arbitrariness choice of $\epsilon$ imply 
that limit set of $t\to \pi(L(t))$ must be contained in $\Sigma$ when $T<\infty$.
Therefore we can also follow the same argument
of Section \ref{subsectionproof-item (a)} and conclude the convergence in the leaf sence, i.e, 
 $\pi(L(t))$ converges to a singular  point of $M/\overline{\F}$, where $\pi:M\to M/\overline{\F}$ is the canonical projection.    

Now we assume that $\F$ is 
infinitesimal compact. 
In order to check the type I convergence, 
let us consider a finite open cover  
$\{ \mathcal{O}_{n}  \}$ of
the compact manifold $B=\overline{L_{q}}$. 
Here $\{ \mathcal{O}_{n} \}$ denotes the  tubular neighborhood
of a plaque $P_n\subset B$ as defined in \cite{AlexandrinoRadeschi-mean-curvature}. 
Note that the discussion introduced in \cite{AlexandrinoRadeschi-mean-curvature}  still holds in the neighborhood $\mathcal{O}_n$
because  $\F$ is infinitesimal compact.

 Since $L(t)$ converges in a leaf sense
to $B$, we can find $t_0$ so that the  integral line $\alpha(t)=\varphi_{t}(p)$ is contained in $\cup_{n} \, \mathcal{O}_{n}$ for $t>t_0.$ 
For  each neighborhood $\mathcal{O}_n$ we know   that there exists an open set  $I_n\subset (t_0,T)$ so that  
$\| A(t)\|\leq \| A\supind{e}(t) \|+\widehat{C}$ for $t\in I_{n}$. 
Here $A\supind{e}(t)$ is the shape operator with respect to the flat metric on $\mathcal{O}_n$, 
see  \cite{AlexandrinoRadeschi-mean-curvature} for details. Also as 
explained in \cite{AlexandrinoRadeschi-mean-curvature},  $\|A\supind{e}(t)\|\sqrt{T-t}<\widetilde{C}$, for $t\in I_{n}$. 
Therefore we have that
$\| A(t)\|\sqrt{T-t}< C_n$ for $t\in I_{n}$. Set $C:=\max_{n} \{ C_n \}.$ Hence    
$\|A(t)\|\sqrt{T-t}<C$ for all $t\in (t_{0},T)$. From this, as explained at the end of Section \ref{revisedproof-typeI}, we infer the type I convergence and that the  integral line
$\alpha$ converges to a point.

\end{proof}

\begin{remark}
 It is possible  to check that there exists a  metric $\metric\supind{0}$
so that $\F|_{U}$ is a generalized isoparametric foliation, when $\F|_{B}$ is a generalized isoparametric foliation (with the respect to 
the original metric).
\end{remark}


\section{Cylinder structure, bounded geometry and MCF }
\label{section-onvergenceMCFviacomparisonlemma}
 
In this section we present a generalization of item (a) of Theorem \ref{theorem-main1}. 

\begin{proposition}
\label{theorem-main3}
Let $\mathcal F$ be a SRF with closed leaves on a complete manifold $(M,g)$. 
Assume that
\begin{enumerate}
    \item   $M$ has bounded sectional  curvature; 
    \item  the shape operator along  each leaf $L \in \mathcal F$ is bounded; 
		\item $M/\F$ is compact. 
\end{enumerate}
Let $N$ be an imersed  submanifold contained in a regular leaf.
Assume that the dimension of $N$ is greater than the dimension of singular leaves, that 
the mean curvature flow  $N(t)$ is a restriction of a basic flow (with respect to 
$\F$) on the regular stratum and  its maximal interval $[0,T)$ has $T <\infty.$ Then  
 $N(t)$ converges to a singular leaf $L$ in the leaf space sense.  
\end{proposition}
   
Again the main idea is to deduce the existence of basins of attraction, i.e, Lemma \ref{newlemma_r_t}. 
The key observation behind the proof of this adaption of Lemma \ref{newlemma_r_t}  is the estimate in 
Lemma \ref{prop:submanifold}, which we believe can be  useful in the context of immersion theory.

\subsection{Comparisons lemmas}

Here  we  
consider two triples $(M_{1},L_{1},\gamma_{1})$ and $(M_{2},L_{2}, \gamma_{2})$, where  
$(M_{i}, \metric_{i})$ is a Riemannian manifold so that $\dim M_1=\dim M_2$,
$L_{i}$ is an embedded submanifold of $M_{i}$ so that $\dim L_1=\dim L_2$ and $\gamma_{i}$ 
is a  unit speed geodesic  orthogonal   to $L_{i}$ at $\gamma_{i}(0)$. We also assume that
$U_{i}$ is a tubular  neighborhood of $L_{i}$ of radius $\epsilon_{0}$ 
so that the distance function  $r_{i}: U_{i}\setminus L_{i}\to \mathbb{R}$ defined as $r_{i}(x)=d_{i}(L_{i},x)$ 
is smooth. In particular we are assuming that  
$\gamma_{i}|_{[0,\epsilon_{0}]}$ does not contain a focal point of $L_{i}$.

It is not difficult to adapt classical arguments of index lemma to conclude the next result,
cf. \cite[Chapter 2, Theorem A]{greenWu}. 

\begin{lemma}\label{lem:comparison}
Assume that:
\begin{enumerate}
\item[(a)] $\sup_{ e_{1}(t)\in  \nu_{1}(s)} K_{1} (e_{1},\gamma'_{1})\leq \inf_{ e_{2}(s)\in  \nu_{2}(t)} K_{2}(e_{2},\gamma_{2}')$
(where $\|e_i\|=1$).  
\item[(b)] $\max_{j} \lambda_{1,j}\leq \min_{j} \lambda_{2,j}$ where $\lambda_{i,j}$ is 
a principal curvature to the shape operator $A_{\gamma_{i}'(0)}.$  
\end{enumerate}
Then there exists an isomorphism $\theta(\epsilon):\nu_{1}(\epsilon)\to \nu_{2}(\epsilon)$ (for $\epsilon<\epsilon_{0}$) so that  
$$ \mathrm{Hess}\, r_{2}(\theta(\epsilon)X_{1},\theta(r)X_{1})\leq  \mathrm{Hess} \, r_{1}(X_{1},X_{1}).$$
Here $\nu_{i}(s)$ denotes the space of normal vectors to $\gamma_{i}'(s).$
\end{lemma}

\begin{remark}
Let  $\mathcal{V}_{i}$ be the vector spaces of differentiable vector fields orthogonal to 
$\gamma_{i}$
starting tangent to $T_{\gamma_{i}(0)}L_{i}.$ The isomorphism defined above is 
an isomorphism between $\mathcal{V}_{i}$. 
In fact  $\theta:\mathcal{V}_{1}\to\mathcal{V}_{2}$ is defined as follows:
let $t\to \{e_{i,\totalMindice}(t)\}$ be the parallel transport along $\gamma_{i}$ of an orthogonal basis where
$e_{i,\sliceindice}\in \nu_{\gamma_{i}(0)}L_{i}$,   
$e_{i,\normalindice}\in T_{\gamma_{i}(0)}L_{i}$  and  
 $e_{0}=\gamma_{i}'(0)$. 
For each $V\in\mathcal{V}_{1}$ written as  $V=\sum_{\totalMindice=1}^{\dim M-1} f_{\totalMindice}e_{1,\totalMindice},$ we set
$\theta(V)=\sum_{\totalMindice=1}^{\dim M-1} f_{\totalMindice}e_{2,\totalMindice}.$  
\end{remark}

We intend to apply the above lemma to compare the submanifold $L$ with a fiber of a warped product. 
During this kind of calculation, we will need to understand what happens with $\theta(X_{1})$ when $X_{1}$ is a vector
perpendicular to a slice $S_{1}$ or when it is tangent to $S_{1}$. This will be related to the next result, which 
roughly speaking,  assures us that, under bounded curvature conditions,  if a parallel
vector field $e_{1,\normalindice_{0}}$ along $\gamma_{1}$ starts tangent to $L_{1}$, then (for small time  $s$) 
the parallel    vector field  $e_{2,\normalindice_{0}}:=\theta(e_{1,\normalindice_{0}})$ has small projection into the slice  $S_2$.   
The next result  is  a direct application of classical Rauch's theorem, see \cite[Chapter 10]{Manfredo}. 
 
\begin{lemma}
\label{lemma-parallelandslices}
Assume that $\sup_{ e_{1}(s)\in  \nu_{1}(s)} K_{1} (e_{1},\gamma'_{1})\leq \inf_{ e_{2}(s)\in  \nu_{2}(s)} K_{2}(e_{2},\gamma_{2}')$
where  $ \nu_{i}(s)$ is the space of normal vectors to $\gamma_{i}'(s)$ and $\|e_i\|=1$. 
Let $e_{i,\normalindice_{0}}$ be a parallel unit vector field along  $\gamma_{i}$ so that 
$e_{i,\normalindice_{0}}(0)\in T_{\gamma_{i}(0)}L_{i}$, for $i=1,2$.
Let $J_{i}$ be a Jacobi field along $\gamma_{i}$ so that $J_{i}(0)=0$,  $J_{i}'(0)\in T_{\gamma_{i}(0)}S_{i}$, 
where $S_{i}:=\exp_{\gamma_{i}(0)}(\nu_{\gamma_{i}(0)}L_{i}\cap B_{\epsilon}(0))$ is a slice at $\gamma_i(0)$
 and
$\|J_{1}'(0)\|=\|J_{2}'(0)\|.$ Then there  exists a constant $C$ such that, for $0<s<\epsilon$, we have:
$$ \big|\metric_{2}\big( e_{2,\normalindice_{0}}(s), J_{2}/s \big) \big| \leq C s^{2}.$$
The constant $C$  depends only on $\sup_{[0,r]} \| R_{2} \|$ and $\sup_{[0,r]} \| J_{1}/s\|$ 
(and in particular does not depend on  frames).
 
\end{lemma}



\subsection{The proof of Proposition \ref{theorem-main3} }

We start by proving a lemma that we believe can be  useful in the context of immersion theory. 

\begin{lemma}
\label{prop:submanifold}
Let $(M,\metric)$ be a Riemannian manifold with bounded sectional curvature, i.e, 
there is a constant $\Lambda \in \mathbb{R}$ such that $|K_\metric| \leq \Lambda$. 
Let  $L$ be an embedded submanifold with  bounded shape operator. Assume that 
there exists a well defined tubular neighborhood $\mathrm{Tub}_{\epsilon_0}(L)$ for some $\epsilon_0$.  
Then,  given a positive integer number $k$, reducing $\epsilon_{0}$ if necessary, 
there exist positive constants  $C, D$, $c_{1}$ and $c_2$  so that for each $\epsilon<\epsilon_{0}$
and for each immersed submanifold  $N\subset \partial \mathrm{Tub}_{\epsilon}(L)$ so that $\dim N=\dim L+k$ we have:
  \begin{equation}
 -\frac{D}{r(x)} - c_2 \leq {\rm tr} (A_{\nabla r}) \leq -\frac{C}{r(x)} + c_1,
\end{equation}
where $A_{\nabla r}$ is the shape operator of 
the immersed submanifold $N\subset \partial \mathrm{Tub}_{\epsilon}(L)$ and 
$r(x) = d(L,x)$ is the distance between $L$ and $x\in N.$
\end{lemma}
\begin{remark}
\
\begin{enumerate}
    \item Lemma \ref{prop:submanifold} can be thought as a natural generalization of  \cite[Chapter 6, Theorem 27]{petersen}.
    \item In the particular case where $N$ coincides with the cylinder, the above lemma gives an estimate
 of mean curvature of cylinders. 
\end{enumerate}
\end{remark}
\begin{proof}[Proof of Lemma \ref{prop:submanifold}]
In what follows we are going to prove that ${\rm tr} (A_{\nabla r}) \leq -\frac{C}{r(x)} + c_1,$ i.e., the part of the lemma that will be used in the proof of Proposition \ref{theorem-main3}.  

We start with a tubular neighborhood $\mathrm{Tub}_{\epsilon_0}(L)$ around $L$ 
such that the map $p \mapsto d(p,L)$ is smooth on $\mathrm{Tub}_{\epsilon_0}(L)\setminus L$. Assume that $\epsilon_0$ is such that $\gamma$ has no focal points to $L$ on $\mathrm{Tub}_{\epsilon_0}(L)$ and denote by $s\to \gamma(s)$ an arc length parametrized geodesic that is perpendicular to $L$ at $s = 0$ and that realizes the distance between $L$ and $x\in N$ i.e, $r(x) = d(x,L) = s$. 

Let $\mathcal{B} := \left\{e_{\totalNindice}\right\}$ be an orthonormal  frame  consisting of tangent vectors to $N$ at $\gamma(s) = x.$ Hence,
\begin{equation}
    -{\rm tr}(A_{\nabla r}) =\sum_{\totalNindice=1}^{\dim N}\metric( \nabla_{e_\totalNindice}\nabla r,e_\totalNindice) 
		= \sum_{\totalNindice=1}^{\dim N}\mathrm{Hess}~r(e_{\totalNindice},e_{\totalNindice}).
\end{equation}

Let $\mathbb{S}^d_{R}$ denote the sphere of radius $R$ with the round metric 
 where $d$ stands for the codimension of $L$ on $M$. 
Let $s\to\gamma_{2}(s)$ be an arc length geodesic in $\mathbb{S}_{R}^{d}$ 
starting at north pole $q_{N}=\gamma_{2}(0) \in \mathbb{S}_{R}^{d}$ and set 
$q_{0}:=\gamma_{2}(-s_{0})\in\mathbb{S}_{R}^{d}.$
Consider $\mathbb{S}_{R}^d \times L$ endowed with the metric of warped product 
\[ \widetilde{\metric}=\metric_{\mathbb{S}_{R}^d}\times e^{2\phi}\metric_L,\]
where $\metric_L$ is the induced metric on $L$ from $\metric$ and
\[\phi(p) := \lambda d^2(p,q_0),~\lambda <0,~q_0 \in \mathbb{S}_R^d.\]

By applying the explicit calculations of sectional curvature of warped product metrics 
presented in \cite[Proposition 2.2.2, pg 59]{gw}, 
and choosing 
$|\lambda|$ and $R$ big enough and 
$q_0$ close enought to
the north pole $q_N \in \mathbb{S}_{R}^d,$ one can check the next claim.

\begin{claim}\label{claim:quepermite}
There exist real numbers $R >0$ and $\lambda<0$
and a point $q_0\in \mathbb{S}_{R}^d$ so that

\[\big(M_{1}:=\left(\tub_{\epsilon_{0}}(L),\metric \right),~L_1 := L , \, \gamma_1:=\gamma \big)\]
\[\big( M_2:=\left(\mathbb{S}_{R}^d\times L, \tilde{\metric}\right), ~L_2:= \{q_N\}\times L,  \, \gamma_2 \big) \]
fulfill all the requirements of Lemma \ref{lem:comparison}. 
\end{claim}


Let $U_2 := \mathrm{Tub}_{\epsilon_0}(L_{2})$  be a smooth tubular neighborhood of $L_{2}$ on $M_2$ and consider the distance function
\[r_2 : U_2\setminus L_{2} \to \mathbb{R},\]
\[r_2 : p \mapsto d_{\widetilde g}(p, L_{2}),\]

For each $\totalNindice \in \{1,\ldots,\dim N\}$, 
\[\theta (e_{\totalNindice}) = e_{\totalNindice}^{\top} + e_{\totalNindice}^{\perp},~e_{\totalNindice} \in \mathcal{B},\]
where $e_{\totalNindice}^{\top}$ denotes the component that is tangent to $L$ and $e_{\totalNindice}^{\perp}$ the component that is tangent to the base of the warped product.

By the comparison Lemma \ref{lem:comparison},
\begin{equation}\label{eq:comparinghess}\mathrm{Hess}~r(e_{\totalNindice},e_{\totalNindice}) \geq \mathrm{Hess}~r_2(e_{\totalNindice}^{\top},e_{\totalNindice}^{\top}) + \mathrm{Hess}~r_2(e_{\totalNindice}^{\perp},e_{\totalNindice}^{\perp}).\end{equation}
The base of the warped product is the round sphere, thus
\begin{equation}
    \mathrm{Hess}~r_{2}(e_{\totalNindice}^{\perp},e_{\totalNindice}^{\perp}) = 
		\dfrac{1}{R}\cot\left(\dfrac{r_2}{R}\right)\|e_{\totalNindice}^{\perp}\|^2,
\end{equation}
furthermore, from \cite[Proposition 2.2.2, pg 59]{gw} 
\begin{equation}
    \mathrm{Hess}~r_{2}(e_{\totalNindice}^{\top},e_{\totalNindice}^{\top}) = 2\lambda e^{2\lambda(s+s_0)^2}(s+s_0)\|e_{\totalNindice}^{\top}\|^2,~\lambda < 0.
\end{equation}
\begin{equation}
    \mathrm{Hess}~r(e_{\totalNindice},e_{\totalNindice}) \geq 2\lambda(\epsilon_0 + s_0)e^{2\lambda(\epsilon_0+s_0)^2}\|e_{\totalNindice}^{\top}\|^2 + \dfrac{1}{R}\cot\left(\dfrac{r_2}{R}\right)\|e_{\totalNindice}^{\perp}\|^2.\nonumber
\end{equation}
Summing up  and using the definition of ${\rm tr}(A_{\nabla r})$ one has
\begin{equation}
     {\rm tr}(A_{\nabla r})\leq - 2\lambda(\epsilon_0 + s_0)e^{2\lambda(\epsilon_0+s_0)^2}\sum_{\totalNindice}\|e_{\totalNindice}^{\top}\|^2 - \dfrac{1}{R}\cot\left(\dfrac{r_2}{R}\right)\sum_{\totalNindice}\|e_{\totalNindice}^{\perp}\|^2.\nonumber
\end{equation}
and hence, we conclude
\begin{equation}
    {\rm tr}(A_{\nabla r})\leq -2\lambda(\epsilon_0 + s_0)e^{2\lambda(\epsilon_0+s_0)^2}\sum_{\totalNindice}\|e_{\totalNindice}^{\top}\|^2 - \frac{1}{R}\cot\left(\dfrac{ r}{R}\right)\sum_{\totalNindice}\|e_{\totalNindice}^{\perp}\|^2.
\end{equation}

Writing  $\cot\left(\dfrac{r}{R}\right) = \dfrac{R}{r} - O(r),$ where $O(r) > 0,$ we have,
\begin{equation}
     {\rm tr}(A_{\nabla r})\leq -2\lambda(\epsilon_0 + s_0)e^{2\lambda(\epsilon_0+s_0)^2}\sum_{\totalNindice}\|e_{\totalNindice}^{\top}\|^2 - \frac{1}{R}\left(\dfrac{R}{r} - O(\epsilon_0)\right)\sum_{\totalNindice}\|e_{\totalNindice}^{\perp}\|^2.
\end{equation}

Since, for every $\totalNindice\in \{1,\ldots,\dim N\},$ one has $\|e_{\totalNindice}^{\top}\|, \|e_{\totalNindice}^{\perp}\| \leq 1,$ we define
\begin{equation}
    c_1 := -2\lambda(\epsilon_0 + s_0)e^{2\lambda(\epsilon_0+s_0)^2}\dim N + \dfrac{1}{R}O(\epsilon_0)\dim N.
\end{equation}
to infer  that
\begin{equation}
    {\rm tr}(A_{\nabla r})\leq c_1 - \dfrac{1}{r}\sum_{\totalNindice}\|e_{\totalNindice}^{\perp}\|^2.
\end{equation}

Now, reducing $\epsilon_0$ if necessary, we  evoke Lemma \ref{lemma-parallelandslices} to conclude that there is $C> 0$ such that

\begin{equation}
\sum_{\totalNindice}\|e_{\totalNindice}^{\perp}\|^2\geq \sum_{\sliceindice}\|e_{\sliceindice}^{\perp}\|^2  \geq C,
\end{equation}
therefore,
\begin{equation}
    {\rm tr}(A_{\nabla r})\leq c_1- \dfrac{C}{r(x)},
\end{equation}
finishing  the proof of the main  part of the lemma 
that will be used in the proof of Proposition \ref{theorem-main3}.  The proof of  
$-\frac{D}{r(x)} - c_2 \leq {\rm tr} (A_{\nabla r})$ can be done in a similar way,
considering a warped metric on $\mathbb{H}_{R}^{d}\times L.$
\end{proof}

From now on, we let $(M,\mathcal F,N,\metric)$ be a setup just as in the hypothesis of Proposition \ref{theorem-main3}. 
Starting by recalling  that the MCF $t\to N(t)$ of $N$   preserves  the dimension of $N(t)$ for $t<T.$
This fact, and the same argument in Section \ref{subsectionproof-item (a)} allow us to reduce the proof of Proposition \ref{theorem-main3} to check again the existence of  basins of attraction, i.e., an adaptation of Lemma \ref{newlemma_r_t}.  
It is a straightforward consequence of Lemma \ref{prop:submanifold} 
with $L = L_q$ and $N = N(t)$ that there exists $\epsilon > 0$ such that 
    \begin{enumerate}
        \item there is a constant $C_1 >0$ depending only on the $\mathrm{Tub}_{\epsilon}(L_q)$ such that if $N(t_0) \in \mathrm{Tub}_{\epsilon}(L_q)$ for some $t_0 > 0,$ then 
        $r^2(t_0) - r^2(t) \geq C_1^2(t-t_0),~\forall t \in [t_0,T),$ $ \, \, \mathrm{and} \, \, T<\infty.$ 
     \item $N(t) \subset \mathrm{Tub}_{\epsilon}(L_q)~ \forall t \in [t_0,T).$
      \end{enumerate}

\section{Appendix: Inclination of the distribution $\mathcal{T}$ }
\label{Appendix-inclination-tau}
In this section we present the proof of
Lemma \ref{bounded-Tau}. Roughly speaking, this lemma and   
 Lemma  \ref{lemma-parallelandslices}  
assure  us  that,
under bounded curvature conditions and small $\epsilon$, 
the ``inclination'' between the distribution   $\mathcal{T}$(defined in Section \ref{section-sasakimetric}) 
and a slice $S_{\epsilon}(\tilde{q})$ is bounded, and this bound does not depend on $\tilde{q}\in L_q.$
\begin{lemma}
\label{bounded-Tau}
Assume that $M$ has bounded curvature and the shape operator of each leaf of $\F$ on $M$ is bounded. 
Then for each leaf $L_q$ and small $\epsilon>0$ 
 there exists a radius $r_0$ with the following property: if
\begin{itemize}
\item  $X$ is a unit vector tangent to $L_{q}$ at a point $\widetilde{q}$,  
\item $\gamma$ is  an unit speed geodesic  orthogonal to $L_{q}$ at 
$\widetilde{q}=\gamma(0),$ 
\item $s\to J(s)$ is the (unique) Jacobi field along $\gamma$ so that $J(0)=X$ and $J(s)\in \mathcal{T}_{\gamma(s)},$
\item $ s\to \overline{e}(s)$ is a parallel vector field along $\gamma$
so that $\overline{e}(0)$ is normal to $L_{q}$ and $\|\overline{e}\|=1$.
\end{itemize}
Then $ | \|J(s)\| -1 |<\epsilon  \, \, \mathrm{and} \, \,  |\metric(J(s),\overline{e})| <\epsilon$ 
for each $s\in [0,r_0]$.  
\end{lemma}
\begin{proof}
 Our strategy consists in bounding the inicial conditions of $J$ by
comparing them with inicial conditions of a  Jacobi field along a geodesic $\gamma_0$ starting at $q$, stressing that they do not depend
on $\widetilde{q}\in L_{q}$.

We first assume that $\widetilde{q}$ and $q$ are in the same  plaque
of $L_q$  and that
exists a  unit vector field $\vec{X}$ on this plaque so that  $\vec{X}(\widetilde{q})=X$.
We can then extend this vector field $\vec{X}$ to be tangent to the leaves near the plaque of $L_q$ and 
linearize this vector field, producing a linearized vector field $\vec{X}^{\ell}$; see \cite{Alexandrino-Inagaki-Struchiner,Alexandrino-Radeschi-Molino}. Let $\varphi$ be  the flow
of  $\vec{X}^{\ell}$. As explained in \cite{Alexandrino-Inagaki-Struchiner,Alexandrino-Radeschi-Molino} this flow has the following properties:
\begin{enumerate}
\item[(a)] The flow $\varphi$ sends fiber to fiber of the normal bundle $\nu(L_q)$;
\item[(b)] $\varphi_{t}: \nu_{p}(L_{q})\to \nu_{\varphi_{t}(p)}(L_{q})$ is a linear isometry;
\item[(c)]  For each $p\in L_{q}$ near $\widetilde{q}$   we have $\vec{X}(p)=\frac{d}{d t}\varphi_{t}(p)|_{t=0}$. 
\end{enumerate}

Let us denote $t_0$ the time so that $\widetilde{q}=\varphi_{t_{0}}(q)$ and $\gamma_0$ the geodesic defined as 
$\gamma=\varphi^{t_{0}}\circ\gamma_0$.  Define the geodesic variations $f(t,s):=\varphi_{t}\circ\gamma_{0}(s)$ and 
the associated Jacobi field $J_{t} (\cdot) :=\frac{\partial}{\partial t} f(t,\cdot).$ Hence $J=J_{t_0}$.
Also note that $J_{t}(s)=d \varphi_{t} J_{0}(s).$ Using property (b) described above, we can infer: 
 \begin{equation}
\label{eq-DerivadaCovariantenormal}
d\varphi_{t_{0}}\Big(\frac{\nabla}{\partial s}J_{0}(0)\Big)^{\nu}= \Big(\frac{\nabla}{\partial s}J_{t_{0}}(0)\Big)^{\nu},
\end{equation} 
where $(\cdot)^{\nu}$ is the normal component. In fact, since $\varphi_{t}: \nu_{p}(L_{q})\to \nu_{\varphi_{t}(p)}(L_{q})$ is a linear isometry
we have that $\frac{\partial }{\partial s} f(t,0)=\exp(t\xi)\frac{\partial }{\partial s} f(0,0)$ where the matrix exponential
$\exp(t\xi)$ is written with respect to some parallel frame (with respect to the normal connection $\nabla^{\nu}$)
along the curve $t\to f(t,0)$. 
 
\begin{eqnarray*}
\Big(\frac{\nabla}{\partial s}J_{t_{0}}(0)\Big)^{\nu}&=& \Big(\frac{\nabla}{\partial s}\frac{\partial}{\partial t} f(t_{0},0)  \Big)^{\nu}\\
 &=&\frac{\nabla}{\partial t}^{\nu}\frac{\partial}{\partial s} f(t_{0},0)  \\
&=& \frac{d}{d t}\exp(t\xi)\frac{\partial }{\partial s} f(0,0)|_{t=t_{0}}\\
&=&\frac{d}{d t}\exp(t_{0}\xi)\exp(t\xi)\frac{\partial }{\partial s} f(0,0)|_{t=0}\\
&=&\exp(t_{0}\xi)\frac{d}{d t}\exp(t\xi)\frac{\partial }{\partial s} f(0,0)|_{t=0}\\
&=& exp(t_{0}\xi)\frac{\nabla}{\partial t}^{\nu}\frac{\partial}{\partial s} f(0,0)  \\
&=& exp(t_{0}\xi) \Big(\frac{\nabla}{\partial s}\frac{\partial}{\partial t} f(0,0)  \Big)^{\nu}\\
&=&d\varphi^{t_{0}}\Big(\frac{\nabla}{\partial s}J_{0}(0)\Big)^{\nu}
\end{eqnarray*}

From Eq. \eqref{eq-DerivadaCovariantenormal} we infer: 
\begin{equation}
\label{eq-2-DerivadaCovariantenormal}
\Big\|\Big(\frac{\nabla}{\partial s}J_{0}(0)\Big)^{\nu}\Big\|= \Big\|\Big(\frac{\nabla}{\partial s}J_{t_{0}}(0)\Big)^{\nu}\Big\|
\end{equation}
On the other hand, note  that each Jacobi field  $\widetilde{J}_0$ along $\gamma_0$ tangent to $\mathcal{T}$ is unique determined 
by $\widetilde{J}_{0}(0)$ and hence, for each of such Jacobi field with $\|\widetilde{J}_{0} (0)\|=1$, we can infer that there exists
a constant $c_1$ such that
\begin{equation}  
\label{eq-2-5-DerivadaCovariantenormal}
\Big\|\Big(\frac{\nabla}{\partial s}\widetilde{J}_{0}(0)\Big)^{\nu}\Big\|\leq \Big\|\frac{\nabla}{\partial s}\widetilde{J}_{0}(0)\Big\|\leq c_1. 
\end{equation}
Also recall that for each Jacobi field $\widetilde{J}_{t}$ along $\gamma_t$ that is $L_q$-Jacobi field one has
\begin{equation}
\label{eq-2-6-DerivadaCovariantenormal}
 \Big(\frac{\nabla}{\partial s}\widetilde{J}_{t}(0)\Big)^{\top}= -A_{\gamma'_{t}(0)} \widetilde{J}_{t}(0), 
\end{equation}
where $(\cdot)^{\top}$ is the tangent component. 
The fact that the shape operator along $L_q$ is bounded by $\|A \|$ and 
Eq. \eqref{eq-2-DerivadaCovariantenormal},  \eqref{eq-2-5-DerivadaCovariantenormal}
and \eqref{eq-2-6-DerivadaCovariantenormal} imply:
\begin{equation}
\label{eq-3-DerivadaCovariantenormal}
\Big\|\Big(\frac{\nabla}{\partial s}J_{t_{0}}(0)\Big)\Big\| \leq
\sqrt{ \|A\|^{2}+ c_{1}^{2}}.
\end{equation}
It is important to note that the constant $c_2:=\sqrt{ \|A\|^{2}+ c_{1}^{2}}$ does not depend on the point $\widetilde{q}$.
We infer then that for  Jacobi fields $\widetilde{J}$ along $\gamma$ with $\|\widetilde{J}(0)\|=1$  so that $\widetilde{J}$ 
are tangent to $\mathcal{T}$  have  inicial conditions bounded by constants that do not depend on  the point $\widetilde{q}$. 

Now consider a parallel frame $\{\overline{e}_{\totalMindice}\}$ along $\gamma$ so that $\overline{e}_{\normalindice}(0)\in T_{\widetilde{q}}L_{q}$
and $\overline{e}_{\sliceindice}(0)\in \nu_{\widetilde{q}}L_{q}.$ Let $y_{\totalMindice}$ the functions so that
 $J_{t}(s)=\sum_{\totalMindice} y_{\totalMindice}(s) \overline{e}_{\totalMindice}.$ Then the Jacobi equation can be written (in this frame) as
$y''(s)+ R(s)y(s)=0$.  Set $\mathcal{R}(q,\dot{q},s)=(\dot{q}, R(s)q,1)\in\mathfrak{X}(\mathbb{R}^{2n+1}).$

\begin{claim}
\label{claim-appendix}
 Consider the flow $\varphi^{\mathcal{R}}$ of the vector field   $\mathcal{R},$ the compact set 
$I:=\mathbb{S}_{1}^{n-k-1}\times \{0\}\times \overline{B_{c_{2}}(0)}\times \{ 0\}\subset \mathbb{R}^{n-k}\times\mathbb{R}^{k}\times\mathbb{R}^{n}\times \mathbb{R}$
that contains the inicial conditions, and  $r_1$ so that $I\subset B_{r_{1}}(0)\times [0,r_{1})\subset \mathbb{R}^{2n}\times\mathbb{R}.$ 
Then we can find a time $s_1$ 
so that $\varphi^{\mathcal{R}}_{s}(x)\subset B_{r_{1}}(0)\times [0,r_{1}) \subset \mathbb{R}^{2n}\times\mathbb{R}$  for  each $s\in[0,s_{1}]$ 
and $x\in I$. 
In particular, since the curvature is bounded and $\overline{B_{r_{1}}}(0)\times [0,r_{1}]$ is compact, there exists a constant $c_3$
(that does not depend on $\widetilde{q}$) 
so that  $|\mathcal{R}\circ\varphi_{s}^{\mathcal{R}}(x)|< c_{3}$. 
 \end{claim}

It follows from Claim \ref{claim-appendix} that
\begin{eqnarray*}
\big|\varphi^{\mathcal{R}}_{s}(x)-\varphi^{\mathcal{R}}_{0}(x)\big| &\leq & \big|\int_{0}^{s} \frac{d}{d t}\varphi_{t}^{\mathcal{R}}(x) dt\big| \\
&\leq & \int_{0}^{s} |\mathcal{R}\circ\varphi^{\mathcal{R}}_{t}(x)| dt\\
&\leq & s c_{3}. 
\end{eqnarray*}
for $x\in I.$ 
Set  $\pi_{1}(q,\dot{q},s)=q.$ The above equation and triangle inequality  imply that there exists $r_0$ so that for $0<s<r_{0}$ 
we have for all $x\in I$
\begin{eqnarray}
\label{eq-4-DerivadaCovariantenormal}
1-\epsilon  \leq  \big|\pi_{1}(\varphi^{\mathcal{R}}_{s}(x)) \big|  \leq  1+\epsilon
\end{eqnarray}

\begin{eqnarray}
\label{eq-5-DerivadaCovariantenormal} 
\big| \langle \varphi_{s}^{\mathcal{R}}(x), e_{\sliceindice}\rangle \big|\leq \frac{\epsilon}{k}.
\end{eqnarray}

Eq. \eqref{eq-4-DerivadaCovariantenormal} and \eqref{eq-5-DerivadaCovariantenormal}
conclude the proof of the lemma in the case where $\widetilde{q}$ is the same plaque of $q$. 
The general case follows direct from compositions of flows of linearized flows.  

\end{proof}
\begin{ack}
The authors are grateful to Marco Radeschi, Miguel Dom\'{\i}nguez V\'{a}zquez, 
Dirk T\"{o}ben, Francisco Caramello and Francisco J. Gozzi  for helpful discussion and suggestions.
\end{ack}



\bibliographystyle{amsplain}

\end{document}